\documentclass[12pt]{amsart}
\usepackage{fullpage,amssymb,url,graphicx}
\makeatletter
\let\@@pmod\pmod
\DeclareRobustCommand{\pmod}{\@ifstar\@pmods\@@pmod}
\def\@pmods#1{\mkern4mu({\operator@font mod}\mkern 6mu#1)}
\makeatother
\newtheorem{theorem}{Theorem}
\newtheorem{lemma}[theorem]{Lemma}
\newtheorem{proposition}[theorem]{Proposition}

\newtheorem{definition}[theorem]{Definition}
\newtheorem{conjecture}[theorem]{Conjecture}
\theoremstyle{remark}
\newtheorem{remarks}[theorem]{Remarks}
\newtheorem{remark}[theorem]{Remark}

\numberwithin{theorem}{section}
\numberwithin{equation}{section}
\numberwithin{table}{section}
\newcommand{\Z}{\mathbb{Z}}
\newcommand{\Q}{\mathbb{Q}}
\renewcommand{\a}{\mathfrak{a}}
\newcommand{\p}{\mathfrak{p}}
\newcommand{\OO}{\mathcal{O}}
\newcommand{\PP}[1]{\mathcal{P}_{#1}}
\DeclareMathOperator{\lcm}{lcm}
\begin{document}
\title{The Euclid--Mullin graph}
\author[A. R. Booker]{Andrew R.~Booker}
\thanks{A.~R.~B.\ was supported by EPSRC Grants {\tt EP/H005188/1},
{\tt EP/L001454/1} and {\tt EP/K034383/1}.}
\address{
Howard House\\
University of Bristol\\
Queens Ave\\
Bristol\\
BS8 1SN\\
United Kingdom
}
\email{\tt andrew.booker@bristol.ac.uk}
\author[S. A. Irvine]{Sean A.~Irvine}
\address{
RTG\\
Level 2\\
18 London St\\
PO Box 9480 WMC\\
Hamilton 3240\\
New Zealand
}
\email{\tt sairvin@gmail.com}
\date{}
\begin{abstract}
We introduce the \emph{Euclid--Mullin graph}, which encodes all instances
of Euclid's proof of the infinitude of primes. We investigate structural
properties of the graph both theoretically and numerically; in
particular, we prove that it is not a tree.
\end{abstract}
\maketitle
\section{Introduction}
The {\em Euclid--Mullin sequence\/} begins [1,] 2, 3, 7, 43, 13, 53,
5, 6221671, where each term is the least prime factor of 1 plus the
product of all the preceding terms. As such it can be viewed as a
computational form of Euclid's proof that the number of primes is
infinite. A companion sequence, sometimes referred to as the {\em
second Euclid--Mullin sequence\/} takes the largest prime factor at
each step. These sequences are {\tt A000945} and {\tt A000946} in the
OEIS~\cite{oeis}. Both sequences were introduced by Mullin~\cite{mullin},
who asked whether every prime occurs in these sequences. Mullin's question
has been answered negatively for the second sequence and in fact the
second sequence omits infinitely many primes~\cite{booker,pollack}. The
question for the first sequence remains open.

Here a generalization is considered, where rather than choosing the least
or largest prime factor at each stage, all prime factors are considered.
Since there are now, in general, multiple choices for the next element,
the result is not a single sequence, but a (directed) graph where each
path from the root to a node corresponds to a particular sequence of
primes. Questions asked about Mullin's sequence can now also be asked
about the graph. In particular, does the graph contain every prime?
If it were ever shown that Mullin's original sequence contains every
prime, then the graph would also include every prime.

The graph admits other structural questions. While the graph is obviously
infinite it would be interesting to know how the number of nodes grows at
each level (or, indeed, to determine if it does grow!).  As a first step
in this direction, this paper establishes that the graph is not a tree.

The directed graph $G_n\subseteq(\Z,\Z\times\Z)$ consists of a set of
integer labelled nodes and edges defined by ordered pairs of nodes. $G_n$
can be defined recursively by: $n$ is a node in $G_n$. If $m$ is a node
in $G_n$, then so are all of $mp_i$ where $m+1=\prod_{i=1}^kp_i^{e_i}$,
$e_i>0$, is the unique factorization of $m+1$. Further, $G_n$ has
directed edges $(m,mp_i)$. It is sometimes convenient to think of the
edge $(m,mp_i)$ as being labelled $p_i$. We say, $n$ is the root of the
graph and has level 0. Any node adjacent to $n$ is said to be level 1. In
general, any node reachable by a directed path of $r$ edges is said to
be level $r$. In fact, a path of length $r$ represents a product of $r$
distinct primes. We call $G_1$ the \emph{Euclid--Mullin graph};
its first few levels are shown in Figure~\ref{fig:g1}.

\begin{figure}
\begin{center}
\includegraphics[width=0.6\textwidth]{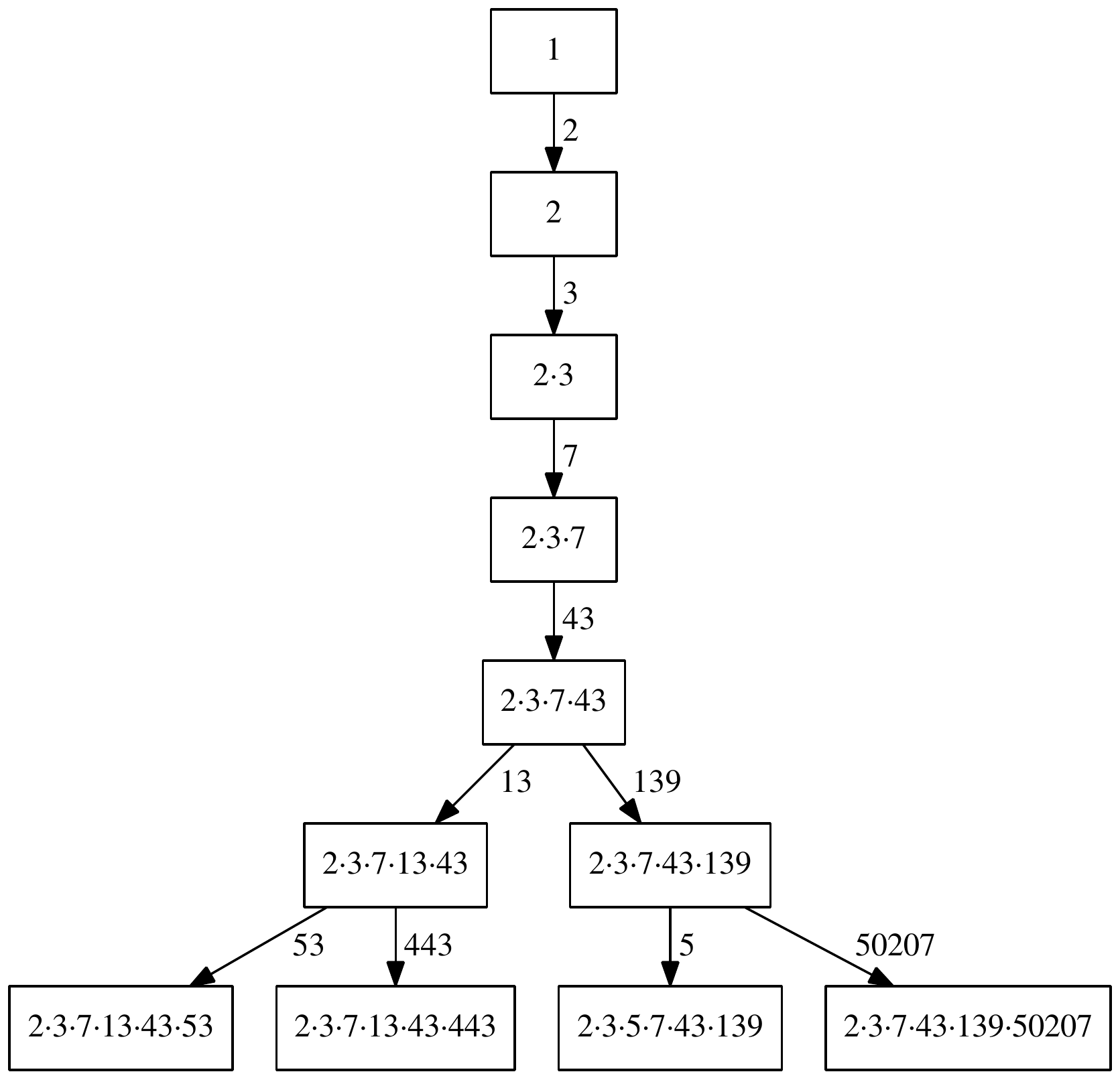}
\end{center}
\caption{\label{fig:g1}$G_1$.}
\end{figure}

\begin{theorem}\label{thm:main}
The Euclid--Mullin graph $G_1$ is not a tree. In particular, each of the following
nodes is connected to $1$ by two distinct paths:
\begin{align*}
2&\cdot 3\cdot 7\cdot 43\cdot 139\cdot 50207\cdot 1607\cdot 38891
\cdot 71609249149971437\cdot 104851\\
&\cdot 5914302068415095755097398828253214149923\\
&\cdot 103\cdot 1750880132687750604376675981842334069\\
&\cdot 103451\cdot 193\cdot 22133\cdot 5587528960270206397663051\\
&\cdot 73\cdot 5\cdot 13\cdot 593
\end{align*}
and
\begin{align*}
2&\cdot 3\cdot 7\cdot 43\cdot 139\cdot 50207\cdot 23\cdot 217733
\cdot 4024572619121\\
&\cdot 539402497343\cdot 72208156847017648587223\cdot 79\\
&\cdot 7269452239696911635939429787229069136737446558564286318153183\\
&\cdot 8689\cdot 107
\cdot 2895777621755988962510175673615781760909999040975810951\\
&\cdot 531543631\cdot 73\cdot 5\cdot 13\cdot 593.
\end{align*}
In each case, the order of the prime factors indicates one path, and the other path
is obtained by swapping $73$ and $593$.
\end{theorem}
Note that the numbers given in the theorem both have level $21$.
Based on some probabilistic considerations presented in
\S\ref{sec:numerics},
we suspect that any node of lower level is connected to
$1$ by a unique path, but answering this definitively is likely to
remain infeasible for the foreseeable future.

\subsection*{Acknowledgements}
The numerical computations for this paper were primarily carried out
using the resources of Real Time Genomics at Hamilton, New~Zealand, with
additional factorizations performed on BlueCrystal phases 2 and 3 at the
University of Bristol, UK, Michael Rubinstein's {\tt riemann} cluster
at the University of Waterloo, Canada, and by NFS@home.  We thank these
organizations for their support, without which this project would not
have been possible. We also thank the users of {\tt mersenneforum.org},
especially Oscar \"Ostlin, who helped with many hard factorizations.

\section{Multiple $k$-tuples of edges}
Given a positive integer $n$, a path in $G_n$ between $n$ and
$m=p_1\cdots p_kn$ can be identified with the $k$-tuple of edge primes
$(p_1,\ldots,p_k)$.
In this section, we formalize this notion and formulate conditions under
which nodes may be connected by more than one path.
We also establish several theoretical results, including the following:
\begin{itemize}
\item
For $k=3$, we obtain a complete classification of the triples
$(p_1,p_2,p_3)$ that form one side of a loop in some $G_n$,
given as the prime values
of certain polynomials; see Theorem~\ref{thm:fibonacci}.
\item We prove that there is a $k\le13$ such that, for any $q\in\Z_{>0}$,
there are infinitely many $k$-tuples $(p_1,\ldots,p_k)$ that form one side
of a loop in some $G_n$ and satisfy $(p_1\cdots p_k,q)=1$.  Moreover,
any given prime occurs as an edge of a loop of height at most $13$
in some $G_n$; see Theorem~\ref{thm:manypairs}.
\end{itemize}

First, let $\PP{k}$ denote the set of $k$-tuples $(p_1,\ldots,p_k)$,
where each $p_i$ is a prime number and $p_i\ne p_j$ for $i\ne j$.
The symmetric group $S_k$ acts on $\PP{k}$ by permuting the indices;
precisely, for $\pi\in S_k$ we write
$\pi.(p_1,\ldots,p_k)=(q_1,\ldots,q_k)$, where $p_i=q_{\pi(i)}$ for
$i=1,\ldots,k$.
\begin{definition}\label{def:equiv}
Let $P=(p_1,\ldots,p_k), Q=(q_1,\ldots,q_k)\in\PP{k}$.
\begin{enumerate}
\item 
We say that $P$ and $Q$ are \emph{equivalent},
and write $P\sim Q$, if there exists
$\pi\in S_k$ such that $Q=\pi.P$ and
$$
p_1\cdots p_{i-1}\equiv q_1\cdots q_{\pi(i)-1}\pmod*{p_i}\quad
\text{\rm for }i=1,\ldots,k.
$$
\item
The \emph{multiplicity} of $P$, denoted $m(P)$, is the number of
$\pi\in S_k$ such that $P\sim\pi.P$.
\item
We say that $P$ is \emph{multiple} if $m(P)>1$.
\item
We call $p_1\cdots p_k$ the \emph{modulus} of $P$, and denote it by
$|P|$.
\end{enumerate}
\end{definition}
It is straightforward to verify that $\sim$ defines an equivalence
relation on $\PP{k}$.  Its relevance to the graphs $G_n$
is described by the following key lemma.

\begin{lemma}\label{lem:NP}
For $P=(p_1,\ldots,p_k)\in\PP{k}$,
let $N(P)$ denote the set of positive integers
$n$ such that $n$ and $|P|n$
are connected in $G_n$ via edges $p_1,\ldots,p_k$, i.e.\
$$p_1\mid{n+1},\quad p_2\mid{p_1n+1},\quad \ldots,
\quad p_k\mid{p_1\cdots p_{k-1}n+1}.$$
Then:
\begin{enumerate}
\item
$N(P)$ is an arithmetic progression modulo $|P|$, i.e.\
$$
N(P)=\{n\in\Z_{>0}:n\equiv a\pmod*{|P|}\}
$$
for some $a=a(P)\in\Z$ relatively prime to $|P|$.
\item
$Q\in\PP{k}$ is equivalent to $P$ if and only if $N(Q)=N(P)$.
\item
For any $n\in N(P)$,
the paths in $G_n$ between $n$ and $|P|n$ are in one-to-one
correspondence with the equivalence class of $P$. In particular, the
number of such paths is the multiplicity $m(P)$.
\end{enumerate}
\end{lemma}
\begin{proof}\hspace{1pt}
\begin{enumerate}
\item
The conditions on $n$ can be rephrased as the system of congruences
\begin{align*}
n&\equiv -1\pmod{p_1}\\
n&\equiv -p_1^{-1}\pmod{p_2}\\
&\;\;\vdots\\
n&\equiv -(p_1\cdots p_{k-1})^{-1}\pmod{p_k},
\end{align*}
and the solutions form an arithmetic progression, by the Chinese
remainder theorem. Since none of the numbers on the
right-hand side can be congruent to $0$, the elements of $N(P)$
lie in an invertible residue class modulo $|P|$.
\item
Suppose that $P=(p_1,\ldots,p_k)$ and $Q=(q_1,\ldots,q_k)$ are
equivalent. Then there is a permutation $\pi\in S_k$ such that
$Q=\pi.P$. Choose $n\in N(P)$, $j\in\{1,\ldots,k\}$, and set
$i=\pi^{-1}(j)$, so that $p_i=q_j$.
Then,
\begin{equation}\label{eqn:pn1qn1}
0\equiv p_1\cdots p_{i-1}n+1\equiv q_1\cdots q_{j-1}n+1
\pmod{p_i=q_j}.
\end{equation}
Since this holds for every $j$, $n$ is contained in $N(Q)$. Since $n$
was an arbitrary element of $N(P)$, this shows that $N(P)\subseteq
N(Q)$. Applying the argument again with the roles of $P$ and $Q$
reversed, we also get $N(Q)\subseteq N(P)$, and hence $N(P)=N(Q)$.

Conversely, suppose that $N(P)=N(Q)$. By part (1), we must have
$|P|=|Q|$, and hence there is a permutation $\pi\in
S_k$ such that $Q=\pi.P$. Let $n\in N(P)=N(Q)$, $i\in\{1,\ldots,k\}$,
and set $j=\pi(i)$, so that $p_i=q_j$. Then again we obtain
\eqref{eqn:pn1qn1}, and since $n$ is invertible modulo
$|P|=|Q|$, it follows that
$$
p_1\cdots p_{i-1}\equiv q_1\cdots q_{j-1}\pmod{p_i=q_j}.
$$
Since this holds for all $i$, $P$ and $Q$ are equivalent.

\item
Let $P=(p_1,\ldots,p_k)$, $n\in N(P)$, and $m=|P|n$. Suppose
that there is a path in $G_n$ between $n$ and $m$ via edges
$q_1,\ldots,q_l$. Then we have $m=q_1\ldots q_ln$, so that
$p_1\cdots p_k=q_1\ldots q_l$. By unique factorization,
we have $l=k$ and $Q=(q_1,\ldots,q_k)\in\PP{k}$.
By part (1), $N(P)$ and $N(Q)$ are arithmetic progressions with the
same modulus.
Since they also have a common element $n\in N(P)\cap N(Q)$,
they must be equal. By part (2), $P$ and $Q$ are therefore equivalent.
Conversely, if $P$ and $Q$ are equivalent then $N(P)=N(Q)$, so there
is a path in $G_n$ between $n$ and $|Q|n=m$.
\end{enumerate}
\end{proof}

\begin{lemma}
There are no multiple $k$-tuples for $k<3$.
\end{lemma}
\begin{proof}
This is obvious for $k=1$. For $k=2$, the only non-trivial possibility
is that $(p_1,p_2)$ is equivalent to $(q_1,q_2)=(p_2,p_1)$. Then by
Definition~\ref{def:equiv} we have
\begin{align*}
1&\equiv q_1=p_2\pmod*{p_1}\\
p_1&\equiv 1\pmod*{p_2},
\end{align*}
so that $p_1<p_2<p_1$, which is impossible.
\end{proof}

\subsection{Multiple triples}
\begin{proposition}\label{prop:triple}
Let $P=(p_1,p_2,p_3)\in\PP{3}$. Then $m(P)>1$ if and only if
\begin{equation}\label{eqn:triple}
p_2(p_1+p_3)\equiv1\pmod*{p_1p_3}
\quad\mbox{and}\quad
p_1\equiv p_3\pmod*{p_2}.
\end{equation}
In this case, $m(P)=2$ and the equivalence class of $P$ is
$\{(p_1,p_2,p_3),(p_3,p_2,p_1)\}$.
\end{proposition}
\begin{proof}
Suppose that $P=(p_1,p_2,p_3)$ is equivalent to
$Q=(q_1,q_2,q_3)=\pi.P$ for some non-trivial $\pi\in S_3$. Since there
are no multiple pairs, we must have $p_1\ne q_1$ and $p_3\ne q_3$, so
$\pi\in\{(13),(123),(132)\}$.

First suppose that $\pi$ is a $3$-cycle.
By reversing the roles of $P$ and $Q$ if
necessary, we may assume that $\pi=(123)$.
Then $(p_1,p_2,p_3)=(q_2,q_3,q_1)$, so
by Definition~\ref{def:equiv} we have
\begin{align*}
1&\equiv q_1=p_3\pmod*{p_1}\\
p_1&\equiv q_1q_2=p_1p_3\pmod*{p_2}
\Longrightarrow 1\equiv p_3\pmod*{p_2}\\
p_1p_2&\equiv 1\pmod*{p_3}.
\end{align*}
Thus, $p_3\equiv1\pmod*{p_1p_2}$ and $p_1p_2\equiv1\pmod*{p_3}$, which
is impossible.

The only remaining choice is $\pi=(13)$. Then
$(p_1,p_2,p_3)=(q_3,q_2,q_1)$, and we have
\begin{align*}
1&\equiv q_1q_2=p_2p_3\pmod*{p_1}\\
p_1&\equiv q_1=p_3\pmod*{p_2}\\
p_1p_2&\equiv 1\pmod*{p_3},
\end{align*}
which is equivalent to the system \eqref{eqn:triple}.
Conversely, the steps above are clearly reversible, so that any
$(p_1,p_2,p_3)$ satisfying \eqref{eqn:triple} is equivalent to
$(p_3,p_2,p_1)$.

Finally, since $(13)$ is the only non-trivial permutation that can
relate equivalent triples, any multiple $P\in\PP{3}$ must have $m(P)=2$
and equivalence class $\{P,(13).P\}$.
\end{proof}

Table~\ref{tab:triple} shows the first few solutions
to \eqref{eqn:triple} with $p_1<p_3$, ordered by modulus.
\begin{table}[h!]
\begin{center}
\begin{tabular}{r|l|l}
$P$ & $|P|$ & $a(P)$\\ \hline
$(2,3,5)$ & $30$ & $19$\\
$(3,2,5)$ & $30$ & $29$\\
$(7,5,17)$ & $595$ & $237$ \\
$(211,197,2969)$ & $123412423$ & $114015537$\\
$(601,577,14449)$ & $5010580873$ & $4793484647$\\
$(8191,8101,737281)$ & $48922495303771$ & $48372940054709$\\
$(22921,21169,276949)$ & $134379711825901$ & $123251758931063$
\end{tabular}
\end{center}
\caption{\label{tab:triple}Multiple triples $P=(p_1,p_2,p_3)$ with $p_1<p_3$}
\end{table}

\subsubsection{Integer triples}
Let us temporarily drop the restriction that $p_1$, $p_2$ and
$p_3$ be prime, and consider all solutions to \eqref{eqn:triple} in
integers. Then it turns out that we can give a complete classification.
In order to state it, we recall that the \emph{Fibonacci polynomials}
$F_n(x)$ are defined by the recurrence
$$
F_0(x)=0,\quad F_1(x)=1,\quad
\mbox{and}\quad F_n(x)=xF_{n-1}(x)+F_{n-2}(x)\quad
\mbox{for }n\ge2,
$$
generalizing the usual Fibonacci numbers $F_n=F_n(1)$.
By convention we extend the definition to negative indices by defining
$F_{-n}(x)=F_n(-x)=(-1)^{n-1}F_n(x)$.
\begin{theorem}\label{thm:fibonacci}
Let $(p_1,p_2,p_3)\in\Z^3$. Then $(p_1,p_2,p_3)$ satisfies
\eqref{eqn:triple} if and only if one of the following holds
for some $n,x\in\Z$ and $\delta\in\{\pm1\}$:
\begin{equation}\label{eqn:parametric}
(p_1,p_2,p_3)=\begin{cases}
\delta(F_{n-1}(x)+F_n(x),F_{-n}(x),F_n(x)+F_{n+1}(x))\\
\delta(F_n(x),F_{-n}(x)+F_{-(n+1)}(x),F_{n+1}(x))\\
\delta(1,x,1)\\
\delta(x,1,1-x).
\end{cases}
\end{equation}
\end{theorem}
\begin{proof}
The Fibonacci polynomials are given by the following explicit formula:
\begin{equation}\label{eqn:binet}
F_n(x)=\frac{\left(\frac{x+\sqrt{x^2+4}}2\right)^n
-\left(\frac{x-\sqrt{x^2+4}}2\right)^n}{\sqrt{x^2+4}}.
\end{equation}
Using this one can verify that
$$
F_{n+1}(x)F_{n-1}(x)=F_n(x)^2+(-1)^n,
$$
and combined with the recurrence
identity $F_{n+1}(x)-F_{n-1}(x)=xF_n(x)$ we see that if
$$
(p_1,p_2,p_3)=\delta(F_{n-1}(x)+F_n(x),F_{-n}(x),F_n(x)+F_{n+1}(x))
$$
then
$$
p_2(p_1+p_3)=1+(-1)^{n-1}p_1p_3
\quad\mbox{and}\quad
p_3-p_1=(-1)^{n-1}xp_2.
$$
Similarly, we obtain the identity
$$
F_{n+1}(x)^2-F_n(x)^2=xF_n(x)F_{n+1}(x)+(-1)^n,
$$
from which it follows that if
$$
(p_1,p_2,p_3)=\delta(F_n(x),F_{-n}(x)+F_{-(n+1)}(x),F_{n+1}(x))
$$
then
$$
p_2(p_1+p_3)=1+(-1)^nxp_1p_3
\quad\mbox{and}\quad
p_3-p_1=(-1)^np_2.
$$
Thus, in either case, $(p_1,p_2,p_3)$ is a solution to \eqref{eqn:triple}.
The final two solutions are straightforward to verify directly.

Now suppose that $(p_1,p_2,p_3)\in\Z^3$ satisfies \eqref{eqn:triple},
and write
\begin{equation}\label{eqn:triple2}
p_3-p_1=qp_2,\quad
p_2(p_1+p_3)=1+rp_1p_3
\end{equation}
for some $q,r\in\Z$.  If $p_1p_2p_3qr=0$ then it is easy to see that
either $p_1p_3=1$ or $p_2(p_1+p_3)=1$, and all such solutions are
described by the third and fourth lines of \eqref{eqn:parametric}.
Otherwise $q$ and $r$ are uniquely determined and non-zero.

Next, set
\begin{equation}\label{eqn:sddef}
s=r(p_1+p_3)-2p_2\quad\mbox{and}\quad
d=(qr)^2+4.
\end{equation}
Then $d$ is not a square, and
a computation shows that $s$ and $p_2$ are related by the
Pell-type equation
\begin{equation}\label{eqn:norm}
s^2-dp_2^2=-4r.
\end{equation}
In other words, $\frac{s+p_2\sqrt{d}}2$ is an element of norm $-r$ in the
quadratic order $\OO=\Z\bigl[\frac{d+\sqrt{d}}2\bigr]$.
(Note that $\OO$ need not be the maximal order in $\Q(\sqrt{d})$.)

If $r=\pm1$ then \eqref{eqn:norm} is just the unit equation for $\OO$. It
is easy to see that $\frac{q+\sqrt{d}}2$ is a fundamental unit (of norm
$-1$), so the general solution of \eqref{eqn:norm} in this case is given by
$$
\frac{s+p_2\sqrt{d}}2=\delta\left(\frac{q+\sqrt{d}}2\right)^n
$$
for $\delta\in\{\pm1\}$ and $n\in\Z$ with $(-1)^{n-1}=r$. Thus,
$$
p_2=\delta\frac{\left(\frac{q+\sqrt{d}}2\right)^n
-\left(\frac{q-\sqrt{d}}2\right)^n}{\sqrt{d}}
=\delta F_n(q)
\quad\mbox{and}\quad
s=\delta\left[\left(\frac{q+\sqrt{d}}2\right)^n
+\left(\frac{q-\sqrt{d}}2\right)^n\right]
=\delta L_n(q),
$$
where $L_n(x)=F_{n+1}(x)+F_{n-1}(x)$ is the Lucas polynomial.
Recalling the definition of $s$, we have
$$
p_1+p_3=\delta'(L_n(q)+2F_n(q)),
$$
where $\delta'=(-1)^{n-1}\delta$.
Together with $p_3-p_1=qp_2=(-1)^{n-1}\delta'qF_n(q)$, this yields
$$
p_1=\delta'\frac{L_n(q)+2F_n(q)-(-1)^{n-1}qF_n(q)}2,
\quad
p_3=\delta'\frac{L_n(q)+2F_n(q)+(-1)^{n-1}qF_n(q)}2.
$$
From the identities
$$
L_n(x)-xF_n(x)=2F_{n-1}(x),\quad
L_n(x)+xF_n(x)=2F_{n+1}(x)\quad\mbox{and}\quad
(-1)^{n-1}F_n(x)=F_{-n}(x),
$$
we get
$$
(p_1,p_2,p_3)=\delta'(F_n(q)+F_{n-1}(q),F_{-n}(q),F_n(q)+F_{n+1}(q))
$$
if $n$ is odd, and
\begin{align*}
(p_1,p_2,p_3)&=\delta'(F_n(q)+F_{n+1}(q),F_{-n}(q),F_n(q)+F_{n-1}(q))\\
&=\delta'(F_{-n}(-q)+F_{-n-1}(-q),F_n(-q),F_{-n}(-q)+F_{-n+1}(-q))
\end{align*}
if $n$ is even. In either case, this is in the form of the first line of
\eqref{eqn:parametric}.

Next suppose that $q=\pm1$. Since
$\frac{r-2+\sqrt{d}}2\in\OO$ has norm $-r$,
we get a family of solutions defined by
\begin{equation}\label{eqn:normrfamily}
\frac{s+p_2\sqrt{d}}2=
\delta\frac{r-2+\sqrt{d}}2\left(\frac{r+\sqrt{d}}2\right)^n
\end{equation}
for $\delta\in\{\pm1\}$ and $n\in2\Z$. Thus,
\begin{align*}
p_2&=\delta\frac{\frac{r-2+\sqrt{d}}2
\left(\frac{r+\sqrt{d}}2\right)^n
-\frac{r-2-\sqrt{d}}2
\left(\frac{r-\sqrt{d}}2\right)^n}{\sqrt{d}}
=\delta\frac{(r-2)F_n(r)+L_n(r)}2\\
&=\delta(F_{n+1}(r)-F_n(r))=\delta(F_{-(n+1)}(r)+F_{-n}(r)),\\
s&=\delta\left[\frac{r-2+\sqrt{d}}2\left(\frac{r+\sqrt{d}}2\right)^n
+\frac{r-2-\sqrt{d}}2\left(\frac{r-\sqrt{d}}2\right)^n\right]\\
&=\delta\frac{(r-2)L_n(r)+dF_n(r)}2
\end{align*}
and
$$
p_1+p_3=\frac{s+2p_2}r=\delta\frac{(r+2)F_n(r)+L_n(r)}2
=\delta(F_{n+1}(r)+F_n(r)).
$$
Combining this with $p_3-p_1=qp_2$, we obtain
$$
(p_1,p_2,p_3)=\delta(F_n(r),F_{-n}(r)+F_{-(n+1)}(r),F_{n+1}(r))
$$
if $q=1$ and
\begin{align*}
(p_1,p_2,p_3)&=\delta(F_{n+1}(r),F_{-n}(r)+F_{-(n+1)}(r),F_n(r))\\
&=\delta(F_{-n-1}(-r),F_{n+1}(-r)+F_n(-r),F_{-n}(-r))
\end{align*}
if $q=-1$. In either case, this is in the form of the second line of
\eqref{eqn:parametric}.

In the case just presented,
it is not obvious that we
obtain all solutions in this manner, but we now proceed to show that
this is indeed the case.
Let us assume first that $4\nmid r$, and
let $\a=\left(\frac{s+p_2\sqrt{d}}2\right)\OO$ be the $\OO$-ideal
associated to the pair $(s,p_2)$.
Then $s+p_2\sqrt{d}\equiv0\pmod*{\a}$, and by \eqref{eqn:sddef} we
have $s+2p_2\equiv0\pmod*{\a}$. It follows from \eqref{eqn:triple2}
that $p_2$ is invertible modulo $r$, so we conclude that
$\sqrt{d}\equiv2\pmod*{\a}$.

Now if $p$ is an odd prime factor of $r$, then
from \eqref{eqn:norm} we
see that $\left(\frac{d}{p}\right)=1$. Thus,
$p\OO$ splits as a product of two prime ideals that are distinguished by
the reduction of $\sqrt{d}$, i.e.\ there is a unique prime ideal
$\p\subseteq\OO$ with norm $p$ such that $\sqrt{d}\equiv2\pmod*{\p}$.

If $r$ is even then $r\equiv2\pmod*{4}$, and from \eqref{eqn:sddef} we
see that $4\mid s$. If $q$ is also even then $d\equiv4\pmod*{16}$, so that
$s^2-dp_2^2\equiv12\pmod*{16}$, in contradiction to \eqref{eqn:norm}. Hence,
$q$ must be odd and $d\equiv8\pmod*{16}$. It follows that the conductor of
$\OO$ is odd and $2$ is ramified in $\Q(\sqrt{d})$, so there is anyway a
unique prime ideal $\p\subseteq\OO$ lying above $2$.

In summary, provided
that $4\nmid r$, we have shown that $r$ is co-prime to the conductor of
$\OO$ and that the prime factors of $\a$ are uniquely determined.
Therefore, any solution of \eqref{eqn:sddef} and \eqref{eqn:norm}
generates the same ideal as
the solution noted above, viz.\ $\left(\frac{r-2+\sqrt{d}}2\right)\OO$.
Hence, \eqref{eqn:normrfamily} describes all solutions.

Next, to handle the case when $4\mid r$ we need to modify the above
argument since the conductor of $\OO$ is even. In this case we set
$$
d'=d/4,\quad
r'=r/4,\quad
s'=s/2\quad\mbox{and}\quad
\OO'=\Z\!\bigl[\tfrac{1+\sqrt{d'}}2\bigr],
$$
and we work over $\OO'$ instead of $\OO$.
Then
$$
d'=4(qr')^2+1\quad\mbox{and}\quad(s')^2-d'p_2^2=-4r',
$$
and if $\a'=\left(\frac{s'+p_2\sqrt{d'}}2\right)\OO'$ then
$N(\a')=|r'|$ and $\frac{1+\sqrt{d'}}2\equiv1\pmod*{\a'}$.
Note that if $r'$ is even then $d'\equiv1\pmod*{8}$, so that
$\left(\frac{d'}2\right)=1$.  Hence,
proceeding as above, for each prime $p\mid r'$,
we find that there is a unique prime $\p\subseteq\OO'$ such that
$N(\p)=p$ and $\frac{1+\sqrt{d'}}2\equiv1\pmod*{\p}$.
Thus, the ideal $\a'$ is again uniquely determined, so
\eqref{eqn:normrfamily} describes all solutions.

It remains only to show that \eqref{eqn:norm} admits no solutions if
$\min(|q|,|r|)>1$. For this we appeal to the reduction theory of
primitive ideals in quadratic orders; see, for instance,
\cite[Chapters 8 and 9]{bv} for terminology and fundamental results.
When $4\nmid r$, we apply the reduction algorithm
to see that the cycle of $\OO$ has length $1$; in other words,
$\OO$ is the only reduced principal $\OO$-ideal.
On the other hand, by \cite[Prop.~9.1.8]{bv},
any primitive $\OO$-ideal of norm less than $\sqrt{d}/2$ is reduced.
Note that if $|q|\ge2$ then
$$
|r|\le\frac{|qr|}2<\frac{\sqrt{(qr)^2+4}}2=\frac{\sqrt{d}}2.
$$
Together these imply that
if $|q|,|r|\ne1$ then there is no primitive,
principal $\OO$-ideal of norm $|r|$, so
\eqref{eqn:norm} is not solvable.

For $r$ divisible by $4$, we similarly apply the reduction algorithm to
$\OO'$ and find that its cycle consists of $\OO'$ together with the ideals
$\left(\frac{qr'-2\pm\sqrt{d'}}2\right)\OO'$ of norm $|qr'|$. In this case
we have $|r'|<\frac12\sqrt{d'}$ for every value of $q$, so there are no
primitive, principal $\OO'$-ideals of norm $|r'|$ if $|q|,|r'|\ne1$.
\end{proof}

\subsubsection{Prime triples}
We now return to the prime case. 
Clearly the third and fourth lines of \eqref{eqn:parametric} never yield
primes, and since the sum of the entries of the second line is
even, the only (positive) prime solutions that it yields are
permutations of $(2,3,5)$.
As for the first line, note that $F_n(x)$ is irreducible only if $|n|$
is prime \cite{levy}. If we take $p_1<p_3$, then we may assume that
$n$ is an odd prime, $x$ is positive, and $\delta=1$.

In particular, with $n=3$ we get the solutions
$$(p_1,p_2,p_3)=(x^2+x+1,x^2+1,x^3+x^2+2x+1).$$
By standard conjectures (Schinzel's
Hypothesis), we expect that these polynomials are simultaneously prime
for infinitely many values of $x>0$, and that motivates the following
conjecture.
\begin{conjecture}
There are infinitely many $P\in\PP{3}$ with $m(P)>1$.
\end{conjecture}
In fact, it is natural to expect triples of primes to occur with
probability proportional to $(\log{x})^{-3}$, so there should be a
constant $c>0$ such that
$$
\#\{P\in\PP{3}:m(P)>1\mbox{ and }|P|<X\}
=(c+o(1))\frac{X^{1/7}}{(\log{X})^3}
\quad\mbox{as }X\to\infty.
$$
Such a statement seems far from what can be proven with present
technology, but we are able to obtain somewhat weaker results in
Section~\ref{sec:largek} below.

\subsection{Multiple quadruples}
In this section we compute the systems of congruences giving rise to
multiple quadruples of edge primes, analogous to
Proposition~\ref{prop:triple} in the case of triples.
Note first that if
$(p_1,p_2,p_3)\sim(p_3,p_2,p_1)\in\PP{3}$ is a multiple triple, then
clearly $(p_0,p_1,p_2,p_3)\sim(p_0,p_3,p_2,p_1)$ and
$(p_1,p_2,p_3,p_4)\sim(p_3,p_2,p_1,p_4)$ are multiple quadruples
for any suitable choice of $p_0$ or $p_4$. More interesting are the
solutions giving rise to loops of height $4$ in the graph. More
generally, we will be interested in pairs
$P=(p_1,\ldots,p_k),Q=(q_1,\ldots,q_k)\in\PP{k}$ defining paths
in $G_n$ that meet only at $n$ and $|P|n=|Q|n$, so
that they form a loop of height $k$;
that is the content of the following definition.
\begin{definition}
Let $P=(p_1,\ldots,p_k),Q=(q_1,\ldots,q_k)\in\PP{k}$. We say that the
pair $(P,Q)\in\PP{k}^2$
is \emph{irreducible} if $P\ne Q$, $P\sim Q$ and
$$
p_1\cdots p_i\ne q_1\cdots q_i
\quad\mbox{for }0<i<k.
$$
\end{definition}
\begin{remark}
Note that $(P,Q)$ is irreducible if and only if $(Q,P)$ is irreducible,
so we may regard the pair as unordered.
\end{remark}

Next, we observe that any equivalence $P\sim Q$ gives rise to another
equivalence, as follows.
\begin{lemma}\label{lem:reverse}
Let $P\in\PP{k}$, and suppose that $P$ is equivalent to $Q=\pi.P$ for
some $\pi\in S_k$. Let
$\sigma=\begin{pmatrix}1&k\end{pmatrix}\begin{pmatrix}2&k-1\end{pmatrix}
\cdots\begin{pmatrix}\lfloor\frac{k}2\rfloor&
k+1-\lfloor\frac{k}2\rfloor\end{pmatrix}\in S_k$
be the permutation that reverses the order of indices, and put
$\widetilde{P}=\sigma.P$, $\widetilde{Q}=\sigma.Q$.
Then:
\begin{enumerate}
\item
$\widetilde{P}$ is equivalent to
$\widetilde{Q}=\sigma\pi\sigma.\widetilde{P}$;
\item
$P$, $Q$, $\widetilde{P}$ and $\widetilde{Q}$ all have the
same multiplicity.
\end{enumerate}
\end{lemma}
\begin{proof}
Suppose that $P=(p_1,\ldots,p_k)$ and $Q=(q_1,\ldots,q_k)$. Then
$$
p_1\cdots p_{i-1}\equiv q_1\cdots q_{j-1}\pmod*{p_i}
$$
whenever $p_i=q_j$. Note that we also have $|P|=|Q|$,
and cancelling the common factor of $p_i=q_j$ yields
$$
(p_1\cdots p_{i-1})(p_{i+1}\cdots p_k)=(q_1\cdots q_{j-1})(q_{j+1}\cdots q_k).
$$
Dividing this equality by the above congruence, we obtain
$$
p_{i+1}\cdots p_k\equiv q_{j+1}\cdots q_k\pmod*{p_i}.
$$
Thus, $(p_k,\ldots,p_1)$ is equivalent to $(q_k,\ldots,q_1)$, as desired.

For the second assertion, $P$ and $Q$ clearly have the same multiplicity
since they are equivalent, and likewise for $\widetilde{P}$ and
$\widetilde{Q}$, so it is enough to show that $m(P)=m(\widetilde{P})$.
But by the first assertion, $P$ is equivalent to $Q$ if and only
if $\widetilde{P}=\sigma.P$ is equivalent to $\widetilde{Q}=\sigma.Q$, so 
$\sigma$ defines a bijection between the equivalence classes of $P$ and
$\widetilde{P}$.
\end{proof}

\begin{proposition}\label{prop:quadruple}
Let $(p_1,p_2,p_3,p_4)\in\PP{4}$.  If the conditions listed in the middle
column of the following table are satisfied in any one case, then each
of the corresponding quadruples in the right column has multiplicity
$2$, with equivalence classes as indicated.  Conversely,
every multiple quadruple has multiplicity $2$,
and every irreducible pair of multiple quadruples
occurs in the table for a unique choice of $(p_1,p_2,p_3,p_4)$.
\begin{center}
\begin{tabular}{|c|c|c|}\hline
& $p_4\equiv1\pmod*{p_1}$ &
$\{(p_1,p_2,p_3,p_4),(p_4,p_1,p_3,p_2)\}$\\
{\rm Case I} & $p_3(p_1p_2+p_4)\equiv1\pmod*{p_2p_4}$ &
$\{(p_4,p_3,p_2,p_1),(p_2,p_3,p_1,p_4)\}$\\
& $p_2\equiv p_4\pmod*{p_3}$ & \\ \hline
& $p_1<p_2$ &
$\{(p_1,p_2,p_3,p_4),(p_4,p_3,p_1,p_2)\}$\\
{\rm Case II} & $p_3(p_1p_2+p_4)\equiv1\pmod*{p_1p_2p_4}$ &
$\{(p_4,p_3,p_2,p_1),(p_2,p_1,p_3,p_4)\}$\\
& $p_1p_2\equiv p_4\pmod*{p_3}$ & \\ \hline
& $p_1<p_4,\quad p_2<p_3$ & $\{(p_1,p_2,p_3,p_4),(p_4,p_2,p_3,p_1)\}$\\
{\rm Case III} & $(p_1+p_4)p_2p_3\equiv1\pmod*{p_1p_4}$ &
$\{(p_4,p_3,p_2,p_1),(p_1,p_3,p_2,p_4)\}$\\
& $p_1\equiv p_4\pmod*{p_2p_3}$ & \\ \hline
& $p_1<p_4$ &
$\{(p_1,p_2,p_3,p_4),(p_4,p_3,p_2,p_1)\}$\\
{\rm Case IV} & $(p_1+p_4)p_2p_3\equiv1\pmod*{p_1p_4}$ &\\
& $p_1\equiv p_3p_4\pmod*{p_2}$ &\\
& $p_1p_2\equiv p_4\pmod*{p_3}$ &\\ \hline
\end{tabular}
\end{center}
\end{proposition}

\begin{remarks}\hspace{1mm}
\begin{enumerate}
\item
Note that the non-trivial permutations of $P=(p_1,p_2,p_3,p_4)$ appearing
in the table are those labelled $Q$, $\widetilde{P}$ and $\widetilde{Q}$
in Lemma~\ref{lem:reverse}; they are all distinct except in Case IV,
where we have $Q=\widetilde{P}$ and $\widetilde{Q}=P$.
\item
The proposition asserts that a given quadruple cannot appear on the
right-hand side of the table more than once,
and that there are never more than two paths
in $G_n$ between $n$ and $p_1p_2p_3p_4n$.
However, it can happen that different permutations
of $(p_1,p_2,p_3,p_4)$ arise from different cases in the table or
from the same case multiple times; for instance, eight
permutations of $(2,3,11,13)$ give rise to quadruples with multiplicity
$2$, and they arise once in Case I and twice in Case IV.
This is not a contradiction because the sets $N(P)$ and $N(P')$ are
disjoint for inequivalent permutations $P$ and $P'$, and thus the
corresponding paths cannot emerge together from the same node.
\item
We will see below that solutions exist in each of the Cases I--IV.
\end{enumerate}
\end{remarks}

\begin{proof}
Let $P=(p_1,p_2,p_3,p_4)$, $Q=(q_1,q_2,q_3,q_4)$, and suppose that
$(P,Q)\in\PP{4}^2$ form an irreducible pair.  Then there is a non-trivial
permutation $\pi\in S_4$ such that $P$ is equivalent to $Q=\pi.P$. Since
$(P,Q)$ is irreducible, $\pi$ cannot stabilize any of the sets $\{1\}$,
$\{1,2\}$ or $\{1,2,3\}$. Moreover, by Lemma~\ref{lem:reverse},
the solutions for a given $\pi$ are in one-to-one correspondence with
those for $\pi^{-1}$, $\sigma\pi\sigma$ and $\sigma\pi^{-1}\sigma$,
where $\sigma=(14)(23)$, so we may group those permutations together
into classes and consider the solutions for only one permutation from
each class.

With some straightforward computations in $S_4$, we find that there are seven
classes:
\begin{equation}\label{eqn:permclasses}
\begin{aligned}
\{(1234),(1432)\}&,\{(1243),(1342)\},\{(13)(24)\},\\
\{(124),(142),(134),(143)\}&,\{(1324),(1423)\},\{(14)\},\{(14)(23)\}.
\end{aligned}
\end{equation}
The first three turn out to yield no solutions, while
the last four correspond to the four cases in the table. We consider
each class in turn and take $\pi$ to be the first element listed
in each case.

\medskip
$\pi=(1234)$:
Then $(p_1,p_2,p_3,p_4)=(q_2,q_3,q_4,q_1)$, and we have
\begin{align*}
1&\equiv q_1=p_4\pmod*{p_1}\\
p_1&\equiv q_1q_2=p_1p_4\pmod*{p_2}
\Longrightarrow 1\equiv p_4\pmod*{p_2}\\
p_1p_2&\equiv q_1q_2q_3=p_1p_2p_4\pmod*{p_3}
\Longrightarrow 1\equiv p_4\pmod*{p_3}\\
p_1p_2p_3&\equiv 1\pmod*{p_4}.
\end{align*}
Thus, we have both $p_4\equiv1\pmod*{p_1p_2p_3}$ and
$p_1p_2p_3\equiv1\pmod*{p_4}$, which is impossible.

\medskip
$\pi=(1243)$:
Then $(p_1,p_2,p_3,p_4)=(q_2,q_4,q_1,q_3)$, and we have
\begin{align*}
1&\equiv q_1=p_3\pmod*{p_1}\\
p_1&\equiv q_1q_2q_3=p_1p_3p_4\pmod*{p_2}
\Longrightarrow 1\equiv p_3p_4\pmod*{p_2}\\
p_1p_2&\equiv 1\pmod*{p_3}\\
p_1p_2p_3&\equiv q_1q_2=p_1p_3\pmod*{p_4}
\Longrightarrow p_2\equiv 1\pmod*{p_4}.
\end{align*}
Thus, $p_1$ divides $1-p_3$, and
$p_2\equiv\frac{1-p_3}{p_1}\pmod*{p_3}$.
Note that applying the permutation $(14)(23)$ to the indices leaves the
system unchanged, so we may assume without loss of generality that
$p_2<p_3$. Therefore, $p_2=p_3+\frac{1-p_3}{p_1}$, whence
$$
p_3\equiv\frac{p_3-1}{p_1}\pmod*{p_2}
\Longrightarrow p_1p_3\equiv p_3-1\pmod*{p_2}.
$$
Since we also have $p_3p_4\equiv1\pmod*{p_2}$, this implies that
$p_1+p_4\equiv1\pmod*{p_2}$.

Now, if $p_2<5$ then we must have $p_2=3$, $p_4=2$, so $p_1>3$
and $p_2<2p_1-3$. On the other hand, if $p_2\ge5$ then $p_2\ge1+2p_4$,
and
$$
1+p_2\le p_1+p_4\le p_1+\frac{p_2-1}2
\Longrightarrow p_2\le 2p_1-3.
$$
Since $p_2=p_3+\frac{1-p_3}{p_1}$, this implies that
$p_3<\frac{2(p_1-\frac32)p_1}{p_1-1}<2p_1$.
Hence, $p_3=p_1+1$, so that $p_1=2$, $p_3=3$. But then $p_2\le 2p_1-3=1$,
which is impossible.

\medskip
$\pi=(13)(24)$:
Then $(p_1,p_2,p_3,p_4)=(q_3,q_4,q_1,q_2)$ and we have
\begin{align*}
1&\equiv q_1q_2=p_3p_4\pmod*{p_1}\\
p_1&\equiv q_1q_2q_3=p_1p_3p_4\pmod*{p_2}
\Longrightarrow 1\equiv p_3p_4\pmod*{p_2}\\
p_1p_2&\equiv1\pmod*{p_3}\\
p_1p_2p_3&\equiv q_1=p_3\pmod*{p_4}
\Longrightarrow p_1p_2\equiv1\pmod*{p_4}.
\end{align*}
Thus, we have both $p_3p_4\equiv1\pmod*{p_1p_2}$ and
$p_1p_2\equiv1\pmod*{p_3p_4}$, which is impossible.

\medskip
$\pi=(124)$:
Then $(p_1,p_2,p_3,p_4)=(q_2,q_4,q_3,q_1)$, and we have
\begin{align*}
1&\equiv q_1=p_4\pmod*{p_1}\\
p_1&\equiv q_1q_2q_3=p_1p_3p_4\pmod*{p_2}
\Longrightarrow 1\equiv p_3p_4\pmod*{p_2}\\
p_1p_2&\equiv q_1q_2=p_1p_4\pmod*{p_3}
\Longrightarrow p_2\equiv p_4\pmod*{p_3}\\
p_1p_2p_3&\equiv 1\pmod*{p_4},
\end{align*}
which is equivalent to the set of conditions in Case I.
The equivalence classes in the right-hand column are
$\{P,\pi.P\},\{\sigma.P,\sigma\pi.P\}$.

\medskip
$\pi=(1324)$:
Then $(p_1,p_2,p_3,p_4)=(q_3,q_4,q_2,q_1)$, and we have
\begin{align*}
1&\equiv q_1q_2=p_3p_4\pmod*{p_1}\\
p_1&\equiv q_1q_2q_3=p_1p_3p_4\pmod*{p_2}
\Longrightarrow 1\equiv p_3p_4\pmod*{p_2}\\
p_1p_2&\equiv q_1=p_4\pmod*{p_3}\\
p_1p_2p_3&\equiv 1\pmod*{p_4},
\end{align*}
which is equivalent to the system of congruences in Case II.
In this case, the system is invariant under the action of
$(12)=\sigma\pi$, but 
the normalization condition $p_1<p_2$ ensures that each set of solutions
$\{P,\pi.P\},\{\sigma.P,\sigma\pi.P\}$ is counted only once.

\medskip
$\pi=(14)$:
Then $(p_1,p_2,p_3,p_4)=(q_4,q_2,q_3,q_1)$, and we have
\begin{align*}
1&\equiv q_1q_2q_3=p_2p_3p_4\pmod*{p_1}\\
p_1&\equiv q_1=p_4\pmod*{p_2}\\
p_1p_2&\equiv q_1q_2=p_2p_4\pmod*{p_3}
\Longrightarrow p_2\equiv p_4\pmod*{p_3}\\
p_1p_2p_3&\equiv 1\pmod*{p_4},
\end{align*}
which is equivalent to the system of congruences in Case III. In
this case, the system is invariant under both $\sigma$ and $\pi$,
but the normalization conditions
$p_1<p_4$ and $p_2<p_3$ ensure that each set of solutions
$\{P,\pi.P\},\{\sigma.P,\sigma\pi.P\}$ is counted only once.

\medskip
$\pi=(14)(23)$:
Then $(p_1,p_2,p_3,p_4)=(q_4,q_3,q_2,q_1)$, and we have
\begin{align*}
1&\equiv q_1q_2q_3=p_2p_3p_4\pmod*{p_1}\\
p_1&\equiv q_1q_2=p_3p_4\pmod*{p_2}\\
p_1p_2&\equiv q_1=p_4\pmod*{p_3}\\
p_1p_2p_3&\equiv 1\pmod*{p_4},
\end{align*}
which is equivalent to the system of congruences in Case IV. In
this case, we have $\pi=\sigma$, so we get only one equivalence class
of solutions.
The system is also invariant under $\pi=\sigma$, but the normalization
condition $p_1<p_4$ ensures that each set of solutions
$\{P,\pi.P\}$ is counted only once.

\medskip
Conversely, it is easy to see that the logic is reversible in the last
four cases considered, so
any $(p_1,p_2,p_3,p_4)\in\PP{4}$ satisfying one of the given
sets of conditions gives rise to multiple quadruples as indicated.

It remains to prove the assertion that the multiplicity is $2$ in each
case. Suppose that $P$ is equivalent to both $Q=\pi.P$ and $Q'=\pi'.P$
for some non-trivial $\pi\ne\pi'$.
Then $Q$ is equivalent to $Q'=\pi'\pi^{-1}.Q$.
Hence, $\pi$, $\pi'$ and $\pi'\pi^{-1}$ are all contained in the union
$$\{(124),(142),(134),(143),(1324),(1423),(14),(14)(23),(13),(24)\}$$
of the last four classes in \eqref{eqn:permclasses}, together with the
permutations giving rise to multiple triples $(p_1,p_2,p_3)$ or
$(p_2,p_3,p_4)$. Note that we are free to replace $P,Q,Q'$ by
$\sigma.P,\sigma.Q,\sigma.Q'$ or to permute them arbitrarily, which is
to say that we can replace $(\pi,\pi')$ by any of the pairs
$$
(\pi,\pi'),\;
(\pi',\pi),\;(\pi^{-1},\pi'\pi^{-1}),\;
(\pi'\pi^{-1},\pi^{-1}),\;
(\pi'^{-1},\pi\pi'^{-1})\;\mbox{ or }\;
(\pi\pi'^{-1},\pi'^{-1}),
$$
or their conjugates by $\sigma$.
Going through all possibilities, we find that we may assume that
$$
(\pi,\pi')\in\{((124),(142)),((13),(124)),((13),(134))\}.
$$
We consider these three cases in turn.

\medskip
$\pi=(124)$, $\pi'=(142)$:
Recall that $\pi=(124)$ leads to the system in Case I.
For $\pi'=(142)$ and $Q'=(q_1',q_2',q_3',q_4')$, we have
$(p_1,p_2,p_3,p_4)=(q_4',q_1',q_3',q_2')$, so that
$p_1\equiv1\pmod*{p_2}$ and $p_1p_2p_3\equiv q_1'=p_2\pmod*{p_4}$.
Hence, $p_2\equiv1\pmod*{p_4}$, and
we also have $p_4\equiv1\pmod*{p_1}$,
so that $p_4<p_2<p_1<p_4$, which is impossible.

\medskip
$\pi=(13)$, $\pi'=(124)$:
Then $(p_1,p_2,p_3,p_4)$ satisfies the system in Case I as well as
\eqref{eqn:triple}. Thus we have
\begin{align*}
1&\equiv p_4\equiv p_2p_3p_4\pmod*{p_1}\\
1&\equiv p_3p_4\equiv p_1p_4\pmod*{p_2}\\
1&\equiv p_1p_2\equiv p_1p_4\pmod*{p_3},
\end{align*}
so that $p_4(p_1+p_2p_3)\equiv1\pmod*{p_1p_2p_3}$. Also,
$p_1p_2p_3\equiv1\pmod*{p_4}$, so that $p_4=\frac{p_1p_2p_3-1}{t}$ for
some $t\in(0,p_1p_2p_3)\cap\Z$. Substituting for $p_4$, we have
$t\equiv-p_1-p_2p_3\pmod*{p_1p_2p_3}$, whence
$t=p_1p_2p_3-p_1-p_2p_3$. Thus,
$$
p_4(p_1p_2p_3-p_1-p_2p_3)=p_1p_2p_3-1,
$$
which implies
$$
p_4p_1-1=((p_4-1)p_1-p_4)p_2p_3\ge6[(p_4-1)p_1-p_4].
$$
Hence $p_1\le\frac{6p_4-1}{5p_4-6}$. If $p_4\ge3$ then this gives
$p_1<2$, while if $p_4=2$ then $2<p_1<3$, but both of these are
impossible.

\medskip
$\pi=(13)$, $\pi'=(134)$:
We have $Q'=(q_1',q_2',q_3',q_4')=(p_4,p_2,p_1,p_3)$, and in view of
\eqref{eqn:triple} we get
\begin{align*}
1&\equiv q_1'q_2'=p_2p_4\pmod*{p_1}\Longrightarrow
p_4\equiv p_2^{-1}\equiv p_3\pmod*{p_1}\\
p_1&\equiv q_1'=p_4\pmod*{p_2}\Longrightarrow
p_4\equiv p_3\pmod*{p_2}\\
p_1p_2&\equiv q_1'q_2'q_3'=p_1p_2p_4\pmod*{p_3}\Longrightarrow
p_4\equiv1\equiv p_1p_2\pmod*{p_3}\\
p_1p_2p_3&\equiv1\pmod*{p_4}.
\end{align*}
Hence $p_4\equiv p_3+p_1p_2\pmod*{p_1p_2p_3}$ and $p_4<p_1p_2p_3$, so
that $p_4=p_3+p_1p_2$. By parity considerations we see that at
least one of $p_1$, $p_2$ and $p_3$ must be $2$, and it follows from
Theorem~\ref{thm:fibonacci} that $(p_1,p_2,p_3)$ is a permutation of
$(2,3,5)$. Therefore, $p_1p_2p_3-1=29$ is prime, so that
$p_4=p_1p_2p_3-1>p_3+p_1p_2$, which is a contradiction.

\medskip
Finally, suppose that a quadruple $P$ occurs in the table for two
different choices of $(p_1,p_2,p_3,p_4)$.
Then, by the above argument, in both instances $P$ must be related to
the other element of its equivalence class by the same permutation.
Thus, either $P$ appears once in each equivalence class in Case II or Case
III, or twice in Case IV. However, the normalization conditions rule
out all of these possibilities.
\end{proof}

Table~\ref{tab:quadruple} shows the first several
solutions to the conditions in Proposition~\ref{prop:quadruple},
ordered by modulus.
\begin{table}[h!]
\begin{tabular}{r|l|l|l}
$(p_1,p_2,p_3,p_4)$ & $|P|$ & $a(P)$ & case\\ \hline
$(2,5,7,3)$ & $210$ & $107$, $149$ & II\\
$(3,13,2,7)$ & $546$ & $181$, $251$ & I\\
$(3,2,11,13)$ & $858$ & $467$, $779$ & I\\
$(11,3,2,13)$ & $858$ & $571$ & IV\\
$(13,3,2,11)$ & $858$ & $857$ & IV\\
$(3,19,11,2)$ & $1254$ & $1127$ & IV\\
$(7,3,2,41)$ & $1722$ & $1721$ & IV\\
$(41,3,2,7)$ & $1722$ & $1147$ & IV\\
$(41,7,2,3)$ & $1722$ & $491$ & IV\\
$(41,7,3,2)$ & $1722$ & $1639$ & IV\\
$(5,29,2,17)$ & $4930$ & $3909$ & IV\\
$(13,2,5,43)$ & $5590$ & $3353$, $5589$ & III\\
$(2,3,31,37)$ & $6882$ & $1183$, $5771$ & II\\
$(3,7,17,89)$ & $31773$ & $22427$, $26966$ & II\\
$(103,31,2,5)$ & $31930$ & $5149$ & IV\\
$(7,23,2,107)$ & $34454$ & $29959$ & IV\\
$(3,17,31,79)$ & $124899$ & $81764$, $81922$ & I\\
$(41,17,2,199)$ & $277406$ & $32635$ & IV\\
$(5,53,37,43)$ & $421615$ & $39559$, $173203$ & II\\
$(73,5,13,593)$ & $2813785$ & $1125513$, $1861426$ & III\\
$(449,67,2,191)$ & $11491706$ & $6517683$ & IV\\
$(241,2,113,3631)$ & $197766046$ & $183764909$, $42003407$ & III\\
$(2,3541,997,103)$ & $727257662$ & $714062125$ & IV\\
$(23,367,401,421)$ & $1425018061$ & $418499259$, $1226476565$ & II
\end{tabular}
\caption{\label{tab:quadruple}Multiple quadruples of small modulus}
\end{table}

\subsection{Multiple $k$-tuples for large $k$}\label{sec:largek}
The alert reader will note that the congruence constraints
in Cases II and III of
Proposition~\ref{prop:quadruple} are nothing but \eqref{eqn:triple}
with $(p_1,p_2,p_3)$ replaced by
$(p_1p_2,p_3,p_4)$ or $(p_1,p_2p_3,p_4)$; in particular, the solutions
are parametrized by Theorem~\ref{thm:fibonacci}.
This turns out to be a general phenomenon, in the sense that the system of
congruences arising from a given element of $S_k$ can be embedded in a
system for any $K>k$ by grouping the primes into products,
as the following lemma shows.
\begin{lemma}\label{lem:embed}
For $i=1,\ldots,k$, let $P_i>1$ be an integer with
prime factors $p_{ij}$ for $j=1,\ldots,r_i$, and assume that
$P_1\cdots P_k$ is squarefree. Put $K=r_1+\ldots+r_k$, and set
$$
P=(p_{11},\ldots,p_{1r_1},\ldots,p_{k1},\ldots,p_{kr_k})\in\PP{K}.
$$
Suppose that $\pi\in S_k$ is a non-trivial permutation such that
$$
P_1\cdots P_{i-1}\equiv Q_1\cdots Q_{\pi(i)-1}\pmod*{P_i}
\quad\mbox{for }i=1,\ldots,k,
$$
where $(Q_1,\ldots,Q_k)=\pi.(P_1,\ldots,P_k)$.
Then there is a non-trivial permutation $\Pi\in S_K$ such that
$P\sim\Pi.P$. Further, the pair $(P,\Pi.P)$ is irreducible if and only
if
$$
P_1\cdots P_i\ne Q_1\cdots Q_i
\quad\mbox{for }0<i<k.
$$
\end{lemma}
\begin{remark}
Note that the order of the prime factors of $P_i$ is not specified, so each solution
$(P_1,\ldots,P_k)$ gives rise to $\prod_{i=1}^kr_i!$ multiple $K$-tuples.
\end{remark}
\begin{proof}
The main idea is to apply $\pi$ to the blocks of indices of length $r_i$.
More formally, for $i=1,\ldots,k+1$, let
$s_i=r_1+\ldots+r_{i-1}$ and
$t_i=r_{\pi^{-1}(1)}+\ldots+r_{\pi^{-1}(i-1)}$.
Note that $s_i+j$ is the index of the $j$th
prime factor of $P_i$ in $P$.
Given $I\in\{1,\ldots,K\}$ we define $\Pi(I)=t_{\pi(i)}+j$, where
$i\in\{1,\ldots,k\}$ and $j\in\{1,\ldots,r_i\}$ are the unique
indices for which $I=s_i+j$.

Note that $\Pi(I)=t_{\pi(i)}+j\le t_{\pi(i)}+r_i\le K$, so
$\Pi$ maps $\{1,\ldots,K\}$ to itself. To see that it defines
an element of $S_K$, it suffices to show that it is surjective. To
that end, given any $I\in\{1,\ldots,K\}$, choose $i$ to be the
largest positive integer such that $t_i<I$, and set $j=I-t_i>0$.
Then $t_i+r_{\pi^{-1}(i)}=t_{i+1}\ge I$, so $j\le r_{\pi^{-1}(i)}$. Hence
$I=\Pi(s_{\pi^{-1}(i)}+j)$, as required.

We must show that $P$ is equivalent to $\Pi.P$.
Let $u_1,\ldots,u_K$ denote the entries
of $P$ and $v_1,\ldots,v_K$ the entries of $\Pi.P$.
Given $I\in\{1,\ldots,K\}$, let
$I=s_i+j$ for $i\in\{1,\ldots,k\}$ and
$j\in\{1,\ldots,r_i\}$. Then
$$
u_1\cdots u_{I-1}
=\left(\prod_{i'=1}^{i-1}P_{i'}\right)\left(\prod_{j'=1}^{j-1}u_{s_i+j'}\right).
$$
Since $u_I\mid P_i$
and $u_{s_i+j'}=v_{t_{\pi(i)}+j'}$ for $j'=1,\ldots,j$,
this is congruent modulo $u_I=v_{\Pi(I)}$ to
$$
\left(\prod_{i'=1}^{\pi(i)-1}Q_{i'}\right)
\left(\prod_{j'=1}^{j-1}v_{t_{\pi(i)}+j'}\right)
=v_1\cdots v_{\Pi(I)-1}.
$$
Since $I$ was arbitrary, $P\sim\Pi.P$.

As for the final claim, if $(P,\Pi.P)$ is not irreducible then
$u_1\cdots u_I=v_1\ldots v_I$ for some $I\in(0,K)\cap\Z$.
If $I<r_1$ then by definition we have
$\Pi(I)=t_{\pi(1)}+I$. Since $u_I$ divides $v_1\cdots v_I$, we also have
$\Pi(I)\le I$. Thus $t_{\pi(1)}=0$, which implies $\pi(1)=1$ and
$P_1=Q_1$. Hence we may assume that $I\ge r_1$.

Let $i<k$ be the largest positive integer such that
$I\ge s_{i+1}$, and $i'<k$ the largest non-negative integer such that
$I\ge t_{i'+1}$.  It follows that
$u_1\cdots u_I$ is divisible by $P_1,\ldots,P_i$ but not by $P_j$ for
any $j>i$. Similarly, $v_1\cdots v_I$ is divisible by
$Q_1\ldots,Q_{i'}$, but not by $Q_j$ for any $j>i'$.
Since $P_j=Q_{\pi(j)}$ for every $j$, it follows that
$\pi$ is a bijection
between $\{1,\ldots,i\}$ and $\{1,\ldots,i'\}$;
hence $i'=i$
and $\pi$ stabilizes $\{1,\ldots,i\}$. In particular,
$P_1\cdots P_i=Q_1\cdots Q_i$.

Conversely, suppose that $P_1\cdots P_i=Q_1\cdots Q_i$ for some
$i\in(0,k)\cap\Z$. We have $P_1\cdots P_i=u_1\cdots u_I$
and $Q_1\cdots Q_i=v_1\cdots v_{I'}$ for some $I,I'\in(0,K)\cap\Z$.
By unique factorization, $I=I'$, and thus $(P,\Pi.P)$ is not
irreducible.
\end{proof}

In the following we let $T_r$ denote the set of squarefree integers
with at most $r$ prime factors, and $T_\infty=\bigcup_{r=0}^\infty T_r$
the set of all squarefree integers.
\begin{lemma}\label{lem:sqfree}
Let $f(x)=(x^2+x+1)(x^2+1)(x^3+x^2+2x+1)$ and
$g(x)=x(x^2-x+1)(x^2+1)$.
Then, for any
$q\in\Z_{>0}$ and all sufficiently large $X>0$ (with the meaning of
``sufficiently large'' possibly depending on $q$), we have
\begin{enumerate}
\item
$
\#\{x\in\Z\cap[1,X]:(f(x),q)=1\mbox{ and }f(x)\in T_\infty\}
\gg_q X;
$
\item
$
\#\{x\in\Z\cap[1,X]:(f(x),q)=1\mbox{ and }f(x)\in T_{13}\}
\gg_q\frac{X}{(\log{X})^3};
$
\item
$
\#\{x\in\Z\cap[1,X]:(g(x),q^2)=q\mbox{ and }q^{-1}g(x)\in T_\infty\}
\gg_q X;
$
\item
$
\#\{x\in\Z\cap[1,X]:(g(x),q^2)=q\mbox{ and }q^{-1}g(x)\in T_{12}\}
\gg_q\frac{X}{(\log{X})^3}.
$
\end{enumerate}
\end{lemma}
\begin{proof}
Let $h\in\Z[x]$ be a squarefree polynomial with $k$
irreducible factors and content $1$, and suppose that
there exists $a\in\Z$ such that
$p\nmid h(a)$ for every prime $p\le\deg{h}$.
Then it was shown in \cite{bb} that if every irreducible
factor of $h$ has degree at most $3$ then there are positive numbers
$c=c(h)$ and $r=r(k,\deg{h})$ such that
$$
\#\{x\in\Z\cap[1,X]:h(x)\in T_\infty\}=(c+o(1))X
\quad\mbox{as }X\to\infty,
$$
and
$$
\#\{x\in\Z\cap[1,X]:h(x)\in T_r\}\gg_h\frac{X}{(\log{X})^k}
\quad\mbox{for }X\gg_h 1.
$$
Further, for $k=3$ and $\deg{h}=7$ we may take $r=13$.
Thus, (1) and (2) follow on applying these results to $h(x)=f(qx)$.

For (3) and (4) we set $Q=\lcm(q,2)$ and take
$h(x)=Q^{-1}g(Q+Q^2x)$. Then $h\in\Z[x]$, and if $a\in\Z$ is such that
$$
Qa\equiv-1\pmod*{\tfrac{15}{(q,15)}},
$$
then $(h(a),30)=1$. From \cite{bb} we find that $r=11$ is admissible for $h$,
from which (3) and (4) follow.
\end{proof}

\begin{theorem}\label{thm:manypairs}\hspace{1pt}
\begin{enumerate}
\item For any $q\in\Z_{>0}$, there are infinitely many positive integers
$k$ such that $\PP{k}^2$ contains an irreducible pair of modulus co-prime
to $q$.
\item There is a positive integer $k\le13$ such that, for any
$q\in\Z_{>0}$, $\PP{k}^2$ contains infinitely many irreducible pairs of
modulus co-prime to $q$.
\item For any squarefree $q\in\Z_{>0}$, there are infinitely many positive
integers $k$ such that $\PP{k}^2$ contains an irreducible pair of modulus
divisible by $q$, and the least such $k$ is at most $\omega(q)+12$.
\end{enumerate}
\end{theorem}
\begin{remark}\hspace{1pt}
\begin{itemize}
\item
Combining (1) and (2) with Lemma~\ref{lem:NP} and the Chinese remainder
theorem, we see that if $a,q,k\in\Z_{>0}$,
then for a positive proportion
of the numbers $n\equiv a\pmod*{q}$, $G_n$ contains both a loop of height
$\le13$ and a loop of height $\ge k$.
If $(a,q)=1$ then the same assertion holds with $n$ restricted to primes.
\item
Similarly, by (3), for any squarefree $q\in\Z_{>0}$ there is a prime $n$
such that $G_n$ contains a loop of height $\le\omega(q)+12$ that has
every prime factor of $q$ as an edge. In particular, every prime occurs
as an edge of a loop in some $G_n$.
\end{itemize}
\end{remark}
\begin{proof}
Let $f(x)$ be as in Lemma~\ref{lem:sqfree}.
Suppose that $f(x)$ is squarefree for some $x\in\Z_{>0}$, and put
$$
(P_1,P_2,P_3)=(x^2+x+1,x^2+1,x^3+x^2+2x+1).
$$
Then the $P_i$ are squarefree and pairwise co-prime. By
Theorem~\ref{thm:fibonacci}, $(P_1,P_2,P_3)$ satisfies \eqref{eqn:triple},
and applying Lemma~\ref{lem:embed} with $\pi=(13)$, we obtain
an irreducible pair $(P,\Pi.P)\in\PP{K}^2$, where $|P|=f(x)$ and
$K=\omega(f(x))$. (Recall that $\omega(n)$ denotes the number of distinct
prime factors of $n$.)

Now, to prove (1), we construct a sequence of positive integers $x_i$ as
follows. Assume that $x_1,\ldots,x_{i-1}$ have been chosen, and set
$$
r=\begin{cases}
0&\mbox{if }i=1,\\
\omega(f(x_{i-1}))&\mbox{if }i>1.
\end{cases}
$$
It was shown by Halberstam \cite{halberstam} that, for any irreducible
polynomial $h\in\Z[x]$,
$\frac{\omega(h(x))-\log\log{x}}{\sqrt{\log\log{x}}}$
has a Gaussian distribution, as in the Erd\H{o}s--Kac theorem. Taking
$h$ to be one of the irreducible factors of $f$,
we have in particular that
$$
\#\{x\in\Z\cap[1,X]:f(x)\in T_r\}
\le\#\{x\in\Z\cap[1,X]:h(x)\in T_r\}
=o(X)\quad\mbox{as }X\to\infty.
$$
Thus, by part (1) of Lemma~\ref{lem:sqfree}, we may choose
$x_i\in\Z_{>0}$ such that $(f(x_i),q)=1$, $f(x_i)$ is squarefree and
$\omega(f(x_i))>r$.

Hence, for the sequence of $x_i$ thus constructed,
$\omega(f(x_i))$ is strictly increasing.  By the
above, for each $i$, $\PP{\omega(f(x_i))}^2$ contains an irreducible
pair of modulus $f(x_i)$, and (1) follows.

Turning to (2), suppose that there is no such $k$. Then for each
$k=1,\ldots,13$, there exists $q_k\in\Z_{>0}$ such that $\PP{k}^2$
contains at most finitely many irreducible pairs of modulus co-prime to
$q_k$, and replacing $q_k$ by a suitable multiple if necessary,
we may assume that there are no such pairs.
Applying part (2) of Lemma~\ref{lem:sqfree} with $q=q_1\cdots q_{13}$,
there exists $x\in\Z_{>0}$ such that $f(x)\in T_{13}$ and $(f(x),q)=1$.
By the above construction, we obtain an irreducible pair 
$(P,\Pi.P)\in\PP{K}^2$ of modulus co-prime to $q$, where
$K=\omega(f(x))\le13$. This is a contradiction, and (2) follows.

Finally, (3) is proved in much the same way using the triple
$$
(P_1,P_2,P_3)=(x,x^2-x+1,x^2+1),
$$
corresponding to the second line of Theorem~\ref{thm:fibonacci} with $n=2$,
and $g(x)$ in place of $f(x)$; we omit the details.
\end{proof}

\subsection{Multiple $k$-tuples with small modulus}
\label{sec:smallmodulus}
One could continue as in Propositions \ref{prop:triple} and
\ref{prop:quadruple} to classify the multiple $k$-tuples for
$k=5,6,\ldots$, but as the proof of Proposition~\ref{prop:quadruple}
shows, this quickly becomes cumbersome. A more practical means
of identifying relatively dense arithmetic progressions $N(P)$ of nodes
giving rise to loops is to do a direct search for small
values of $|P|$.

One procedure for finding all multiple $k$-tuples of a given modulus
is as follows. Suppose that $m$ is a squarefree positive
integer (our candidate for $|P|$), and
rewrite the system of congruences in
Definition~\ref{def:equiv} as
\begin{equation}\label{eqn:recast}
p_1\cdots p_{i-1}\equiv d_i\pmod*{p_i},
\end{equation}
where $d_1,\ldots,d_k$ are proper divisors of $m$ satisfying
\begin{equation}\label{eqn:divisibility}
d_i\ne d_j\mbox{ and }\min(d_i,d_j)\mid\max(d_i,d_j)
\end{equation}
for all $i\ne j$.
(If we wish to find only irreducible pairs, then we impose the further
constraint $d_i\ne p_1\cdots p_{i-1}$.)
We search for solutions to \eqref{eqn:recast} recursively:
suppose that $p_1,\ldots,p_{i-1}$ and $d_1,\ldots,d_{i-1}$ have been
chosen, loop over all proper divisors $d_i$ of $m$ such that
\eqref{eqn:divisibility} holds for all $j<i$, and then over
all primes $p_i\mid\frac{m}{p_1\cdots p_{i-1}}$
such that \eqref{eqn:recast} holds.
Since \eqref{eqn:recast} is a very restrictive condition,
most branches of the search tree are pruned quickly,
so this method is substantially more efficient than naively trying
all permutations of the prime factors of $m$.

We coded this procedure and used it to find $195167$ (unordered)
irreducible pairs of modulus $|P|<10^9$.  The results reveal that for
large $k$, topologies that are much more intricate than the simple loops
observed in Propositions~\ref{prop:triple} and \ref{prop:quadruple}
can arise.  For instance, for any $n\equiv58183403\pmod*{635825190}$,
$G_n$ has a subgraph as shown in Figure~\ref{fig:Gn}, in which there are
$7$ paths between $n$ and $635825190n$, $12$ out of the $21$ pairs of
paths are irreducible, and there are subloops of heights $3$, $4$, $5$,
$6$ and $8$.
\begin{figure}[h!]
\begin{center}
\includegraphics[width=0.8\textwidth]{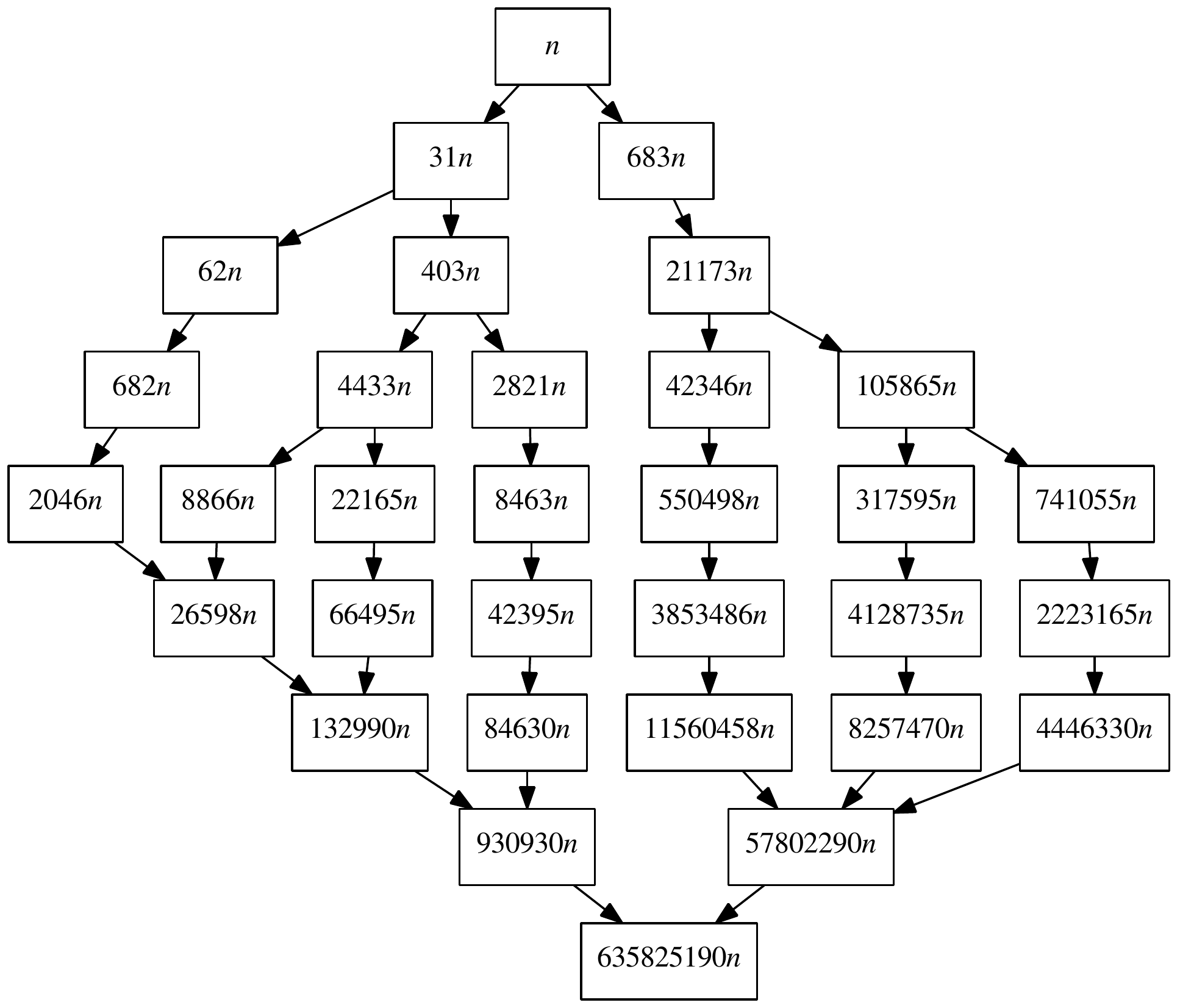}
\end{center}
\caption{\label{fig:Gn}Some nodes of $G_n$ for
$n\equiv58183403\pmod*{635825190}$}
\end{figure}

Note that only pairs of modulus co-prime to $2\cdot3\cdot7\cdot43$
can possibly appear in $G_1$. Imposing that restriction reduces the
list to just $18$ moduli $|P|<10^9$ with $42$ associated arithmetic
progressions $N(P)$, as shown in Table~\ref{tab:smallmodulus}.
Consider, for instance, the progressions with modulus
$115908845=5\cdot13\cdot23\cdot31\cdot41\cdot61$. It is known (see the
introduction of \cite{booker}) that none of these primes can occur as
an edge of the right-most branch of $G_1$ (sequence {\tt A000946}).
Therefore, it seems natural to expect the nodes of the right-most branch
to vary randomly among the invertible residue classes mod $115908845$
as the level increases, in the sense that each residue class should
occur with equal frequency. (This is the same heuristic reasoning as that
supporting Shanks' conjecture \cite{shanks} that the first Euclid--Mullin
sequence contains every prime.) Thus, we would expect one of
the four corresponding residue classes in Table~\ref{tab:smallmodulus}
to occur with frequency $4/\varphi(115908845)=1/19008000$. In
particular, we are led to the following conjecture:
\begin{conjecture}
$G_1$ contains infinitely many loops.
\end{conjecture}

\begin{table}[h!]
\begin{tabular}{l|l|l|l}
$\{p_1,\ldots,p_k\}$ & $|P|$ & $a(P)$ & $1/\mathrm{density}$\\ \hline
$\{5,13,73,593\}$ & $2813785$ & $1125513$, $1861426$ & $1022976$\\
$\{5,11,13,79,523\}$ & $29541655$ & $2913109$, $19876614$ & $9771840$\\
$\{5,13,17,53,563\}$ & $32972095$ & $45473$, $14501753$, $15173846$, $15665474$ & $5611008$\\
$\{5,11,23,31,1307\}$ & $51254005$ & $29374824$, $37354844$ & $17239200$\\
$\{5,13,23,31,41,61\}$ & $115908845$ & $30432518$, $43262953$, $74975328$, $87805763$ & $19008000$\\
$\{197,211,2969\}$ & $123412423$ & $114015537$ & $122162880$\\
$\{5,11,23,67,1831\}$ & $155186405$ & $92870549$ & $106286400$\\
$\{5,13,19,23,71,89\}$ & $179491195$ & $106001778$, $120823468$, $140339224$, $145796156$ & $29272320$\\
$\{5,13,29,61,1597\}$ & $183631045$ & $16718992$, $26947777$, $40801752$, $51030537$ & $32175360$\\
$\{5,11,13,41,73,113\}$ & $241819435$ & $31106978$, $108457851$ & $77414400$\\
$\{5,11,13,17,97,233\}$ & $274715155$ & $161397329$, $273388114$ & $85524480$\\
$\{5,11,13,733,773\}$ & $405125435$ & $26064013$, $332551556$ & $135624960$\\
$\{5,11,19,41,83,127\}$ & $451629145$ & $16717776$, $363759119$ & $148780800$\\
$\{5,13,19,23,37,449\}$ & $471892265$ & $331562178$, $399028904$ & $153280512$\\
$\{5,19,53,337,421\}$ & $714350695$ & $171690041$, $516232304$ & $264176640$\\
$\{5,11,19,53,71,199\}$ & $782534665$ & $504298018$, $599617009$ & $259459200$\\
$\{11,19,29,71,1871\}$ & $805149301$ & $159790883$, $664072158$ & $329868000$\\
$\{5,11,23,61,67,167\}$ & $863399185$ & $132876677$, $201396989$ & $289238400$
\end{tabular}
\caption{\label{tab:smallmodulus}Multiple $k$-tuples of modulus
$|P|<10^9$ with $(|P|,2\cdot3\cdot7\cdot43)=1$}
\end{table}

More generally, it seems likely that each of the residue classes
$N(P)$ in Table~\ref{tab:smallmodulus} will be met infinitely often
by the nodes of $G_1$; we provide some evidence towards this in the
next section.  It is difficult to compute the overall probability
of a random node on the graph landing in one of the residue classes,
since these events are not independent, i.e.\ the classes overlap in
non-trivial ways. However, it is apparent from the first few lines of
the table that the greatest chance of finding a loop comes from the
progressions of modulus $2813785=5\cdot13\cdot73\cdot593$, with density
$2/\varphi(2813785)=1/1022976$.  Thus, on the sub-graph of nodes co-prime
to $2813785$, we expect roughly one out of every million nodes to be
the base of a loop of height $4$.

\section{Numerical results}\label{sec:numerics}
We used two methods for exploring $G_1$ numerically. First, we used freely
available software implementations of the elliptic curve method (see {\tt
GMP-ECM} \cite{gmp-ecm}) and general number field sieve (see {\tt YAFU}
\cite{yafu}, {\tt msieve} \cite{msieve} and {\tt GGNFS} \cite{ggnfs})
to compute as many nodes as was practical for levels up to 17. This was
a community effort, with support from users of {\tt mersenneforum.org}.

\begin{table}[h!]
\begin{center}
\begin{tabular}{rrr|rrr}
level&nodes&composites&level&nodes&composites\\
\hline
$\le4$&1&0 & 11&     555&      0\\
5&2&0      & 12&    2020&      0\\
6&4&0      & 13&    7948&      1\\
7&9&0      & 14&   32738&      8\\
8&24&0     & 15&  141619&    636\\
9&52&0     & 16&  622317&  13445\\
10&165&0   & 17& 2550301& 186060\\
\end{tabular}
\end{center}
\caption{\label{tab:g1}Number of nodes by level in $G_1$.}
\end{table}

Table~\ref{tab:g1} lists the number of nodes that we have computed at
each level of the graph $G_1$. The final column is the number of remaining
unfactored composites at that level.  Factoring a composite at a given
level will increase the number of nodes at that level (by at least 2)
and all subsequent levels.  The single remaining composite at level 13
is the 253-digit:
\begin{quotation}
\noindent\scriptsize 3\-0\-7\-4\-1\-6\-3\-8\-0\-4\-1\-2\-6\-3\-7\-5\-7\-3\-0\-9\-6\-0\-0\-4\-6\-0\-0\-0\-0\-0\-6\-4\-1\-0\-7\-0\-3\-2\-9\-9\-8\-6\-0\-4\-9\-1\-0\-5\-2\-5\-1\-5\-3\-9\-9\-3\-5\-2\-2\-0\-4\-3\-8\-9\-4\-9\-4\-5\-6\-5\-4\-2\-4\-6\-2\-2\-7\-6\-8\-9\-3\-1\-0\-8\-0\-6\-0\-5\-6\-5\-2\-5\-7\-9\-8\-3\-2\-7\-4\-8\-9\-1\-5\-8\-7\-9\-8\-6\-5\-5\-1\-9\-9\-9\-3\-6\-6\-9\-1\-6\-1\-3\-1\-4\-9\-5\-1\-6\-4\-9\-5\-9\-3\-7\-6\-3\-2\-4\-5\-4\-6\-4\-9\-9\-5\-9\-6\-6\-6\-2\-7\-3\-0\-8\-1\-9\-9\-5\-3\-4\-4\-6\-8\-6\-0\-7\-1\-8\-4\-3\-8\-4\-7\-4\-4\-2\-5\-7\-5\-7\-3\-6\-8\-5\-6\-8\-3\-2\-2\-1\-6\-1\-1\-4\-4\-0\-2\-0\-2\-8\-0\-6\-2\-2\-2\-7\-2\-5\-7\-2\-7\-0\-8\-3\-2\-2\-4\-7\-5\-6\-0\-1\-0\-6\-3\-5\-1\-6\-4\-7\-0\-0\-1\-4\-4\-4\-9\-9\-2\-2\-5\-5\-1\-2\-7\-9\-9\-3\-4\-3\-8\-0\-7.
\end{quotation}
We have been unable to factor this number despite running {\tt GMP-ECM}
on approximately 225,000 curves at $B_1=8.5\times 10^8$ and
44,000 curves at $B_1=3\times 10^9$ (and default $B_2$ values); this is
a level of effort comparable to a ``{\tt t70}'', meaning it has a
reasonable chance of revealing any prime factors with up to 70 digits.

The numbers appearing in Theorem~\ref{thm:main} were found by
checking our data for the residue classes in
Table~\ref{tab:smallmodulus}.
If there is a lower node with multiple
paths to $1$ then there must be a loop of height $k$ starting from some
node of level $\le 20-k$, viz.\ at most $17$ if $k=3$ and $16$ if $k\ge4$.
Although we have not been able to compute the full graph up to level $17$,
we expect that there are no more than one million nodes remaining to
be found up to that level, with at most $50$ thousand of those at level $16$
or below (and fewer still that are co-prime to $5$ and $13$).
In view of Table~\ref{tab:smallmodulus}, it seems unlikely
that one of those will yield a loop. However, that cannot be established
definitively until the full graph is computed up to level $18$, which
is out of reach with present technology.

Our second numerical method aimed to produce large quantities of nodes
rather than a comprehensive list of all of them. We began with our
list of nodes at level 17 and followed only the edges corresponding to
primes below some bound $B$.  Taking $B=2^{24}$ and computing up to level
$50$, we found at least one match to every congruence class listed in
Table~\ref{tab:smallmodulus}; in particular we found loops of heights
$3$, $4$, $5$ and $6$. This method was also helpful for investigating
some other statistical questions, as we describe in the next section.

\section{Related questions}
In this final section, we record some numerical observations and
heuristics on related questions:
\begin{itemize}
\item\emph{Does every prime occur as an edge in $G_1$?}
This seems very likely. With the second method described above,
we verified that every prime below $10^9$ occurs.

\begin{table}
\begin{center}
\begin{tabular}{rl}
$k$&$\frac{X_{k+1}/X_k}{\sqrt{2k}}$\\\hline
$9$&$0.748$\\
$10$&$0.752$\\
$11$&$0.776$\\
$12$&$0.803$\\
$13$&$0.808$\\
$14$&$0.825$\\
\end{tabular}
\end{center}
\caption{\label{tab:Xk}Estimated values of
$\frac{X_{k+1}/X_k}{\sqrt{2k}}$}
\end{table}

\item\emph{How does the number of nodes at level $k$ grow
asymptotically as $k\to\infty$?}
Let $X_k$ denote the number of nodes of $G_1$ of level $k$.
Heuristically, based on the Erd\H{o}s--Kac theorem, we expect that
for a typical node $n$, $n+1$ will have about $\log\log{n}$ prime
factors, with the values of $\frac{\log\log{p}}{\log\log(n+1)}$ uniformly
distributed on $[0,1]$ as $p$ varies over the prime factors of $n+1$.

Let $n_k$ be the nodes of a typical path in $G_1$, with
$n_0=1$, and define $\theta_k$ so that
\begin{equation}\label{eqn:thetak}
\frac{n_{k+1}}{n_k}=\exp\bigl([\log(n_k+1)]^{\theta_k}\bigr).
\end{equation}
Then by the above heuristic, $\theta_k$ should vary uniformly over
$[0,1]$ as $k\to\infty$. If we instead treat the $\theta_k$ formally
as independent, uniform random variables on $[0,1]$ and define $n_k$ by
\eqref{eqn:thetak}, then it is not hard to see that
$$
\lim_{k\to\infty}\frac{\log\log{n_k}}{\sqrt{2k}}=1
$$
holds almost surely. Thus, we might expect the typical node of level $k$
to be of size $\exp\exp([1+o(1)]\sqrt{2k})$.  (This analysis ignores
the fact that $n_k+1$ is co-prime to $n_k$ and hence typically has no
small prime factors; however, in the random model, the bulk of the
contribution to $\log\log{n_k}$ comes from the values of $\theta_k$
close to $1$, corresponding to the large prime factors, so this
makes little difference.)  In turn, this leads to the conjecture
that $\frac{X_{k+1}}{X_k}=(1+o(1))\sqrt{2k}$, or equivalently
$\log X_k=\frac{k}2\bigl(\log\frac{2k}{e}+o(1)\bigr)$.

As far as we are aware, it is not even known that $X_k$ is unbounded,
so this remains largely guesswork. Table~\ref{tab:Xk} shows estimated
values of $\frac{X_{k+1}/X_k}{\sqrt{2k}}$
for $k\le 14$, based on the data in Table~\ref{tab:g1}. Although
the data are very limited, they are at least consistent with the above
guess, in that the ratio appears to grow slowly towards $1$.

\item\emph{Are there arbitrarily long chains of nodes with only one
child each?}
This is related to the previous two questions.
The basic heuristic underlying Shanks' conjecture is that the nodes
of a given path in $G_1$ should vary randomly among the invertible
residue classes modulo a fixed prime $p$, until $p$ occurs as an edge
(beyond which every node is divisible by $p$). One (perhaps the only)
conceivable way in which this heuristic might fail is if $n+1$ is prime
for every node $n$ of sufficiently large level along the path. In fact,
as discovered by Kurokawa and Satoh \cite{ks}, that \emph{can} happen
for the analogous question over $\mathbb{F}_p[x]$.

All numerics to date indicate that this pathology does not occur over
$\Z$, but it is an interesting question whether there are arbitrarily
long chains in $G_1$ of nodes $n$ such that $n+1$ is prime.
For a random node $n$, we can estimate the probability that $n+1$ is
prime as $\frac{n}{\varphi(n)\log{n}}$, so the chance
that there is a unique path of length $\ell$ descending from $n$ is
roughly
$$
\frac{n}{\varphi(n)\log{n}}\times
\frac{n}{\varphi(n)\log(n^2)}\times\cdots\times
\frac{n}{\varphi(n)\log(n^{2^{\ell-1}})}
=\left(\frac{n}{\varphi(n)2^{\frac{\ell-1}2}\log{n}}\right)^\ell
\gg_\ell(\log{n})^{-\ell}.
$$
As above, we expect the $n$ of level $k$ typically satisfy
$\log{n}=e^{O(\sqrt{k})}$.  Hence, if our asymptotic guess for $X_k$
holds then we should indeed expect chains of length $\ell$ to occur for
sufficiently large $k$, and in fact we might expect $\ell$ as large as
about $\sqrt{\frac{k\log{k}}{\log2}}$.
By our second method we found several examples
of nodes followed by a unique path of length $4$; the lowest (after the
root node $1$) is the following node at level $20$:
\begin{align*}
2&\cdot3\cdot7\cdot43\cdot139\cdot50207\cdot1607\cdot38891
\cdot71609249149971437\cdot97272377313541\cdot318004829\\
&\cdot1555110880896883\cdot39807662109343\cdot53437\cdot35251\cdot79
\cdot2011283825921\cdot29\cdot17\cdot241.
\end{align*}

\item\emph{Is $G_1$ planar?}
Our search for multiple $k$-tuples of small modulus uncovered several
arithmetic progressions of $n$, e.g.\ $n\equiv93397\pmod*{510510}$, such
that $G_n$ is not planar. As a generalization of Theorem~\ref{thm:main},
it is a natural question whether $G_1$ itself is planar. However, despite
making an extended search, every progression that we found leading to
non-planar graphs had modulus divisible by $6$, and it is unclear whether
or not that is a necessary condition. In any case, if $G_1$ is non-planar,
that fact is likely not manifested until astronomically large level,
so this question is unlikely to be settled in the near future.
\item\emph{How does the number of irreducible pairs of modulus $\le X$ 
grow asymptotically as $X\to\infty$?}
The proof of Theorem~\ref{thm:manypairs} shows that, for large $X$,
\begin{equation}\label{eqn:modasymp}
\#\{q\in\Z\cap[1,X]:\exists\mbox{ an irreducible pair of modulus }q\}
\gg X^{1/5},
\end{equation}
and this gives a lower bound for the number of irreducible pairs of
modulus up to $X$. However, a log-log fit of our data up to $10^9$
suggests that this is too low, and that \eqref{eqn:modasymp} is perhaps
asymptotic to $cX^{5/8}$ for some $c>0$.  Note that the moduli exhibited
in the lower bound in \eqref{eqn:modasymp} are all even (for odd moduli
the proof of Theorem~\ref{thm:manypairs} gives only a lower bound $\gg
X^{1/7}$); our numerics also suggest that almost all irreducible pairs
have even modulus.
\end{itemize}

\begin{table}[h!]
\begin{center}
\begin{tabular}{rrrrl|rrrrl}
$n$&$p$&step&digits&OEIS&$n$&$p$&step&digits&OEIS\\
\hline
2&41&52&335&\tt A000945 & 47&23&36&194&\tt A051319\\
5&31&58&347&\tt A051308 & 53&71&92&526&\tt A051320\\
11&29&56&313&\tt A051309 & 59&37&79&1059&\tt A051321\\
13&17&58&353&\tt A051310 & 61&29&47&501&\tt A051322\\
17&37&31&232&\tt A051311 & 67&19&43&200&\tt A051323\\
19&43&73&922&\tt A051312 & 71&79&140&991&\tt A051324\\
23&29&62&515&\tt A051313 & 73&83&131&949&\tt A051325\\
29&67&80&566&\tt A051314 & 79&17&32&292&\tt A051326\\
31&29&38&240&\tt A051315 & 83&71&65&296&\tt A051327\\
37&59&77&826&\tt A051316 & 89&79&79&743&\tt A051328\\
41&43&56&933&\tt A051317 & 97&53&52&261&\tt A051330
\end{tabular}
\end{center}
\caption{\label{tab:em1}Summary of $M_n$ for $n<100$.}
\end{table}

Finally, we record the latest results on the computation of the
Euclid--Mullin sequence and some of its relatives. Let $M_n$ denote
the first Euclid--Mullin sequence starting with the prime $n$, i.e.\
the edges of the left-most path in $G_n$.  Wagstaff \cite{wagstaff}
computed $M_2$ up through the 43rd term (180 digits). Much computation
effort, including several large GNFS world-wide distributed efforts,
has since been expended on factoring the integers needed to extend
the sequence. In 2012, Ryan Propper found a 75-digit factor using ECM;
it remains the fifth largest factor ever produced by ECM.

Table~\ref{tab:em1} is a summary of known computational results for the
distinct sequences with $n<100$. The `$p$' column is the smallest prime
not yet confirmed as a member of the corresponding sequence. The `step'
column indicates the number of known terms and the `digits' column the
number of decimal digits in the unfactored composite needed for the next
step. The final column is the corresponding entry number in the OEIS.\@
It is unlikely that any of the blocking composites has a factor of less
than 45 digits.

\bibliographystyle{amsplain}
\providecommand{\bysame}{\leavevmode\hbox to3em{\hrulefill}\thinspace}
\providecommand{\MR}{\relax\ifhmode\unskip\space\fi MR }
\providecommand{\MRhref}[2]{%
  \href{http://www.ams.org/mathscinet-getitem?mr=#1}{#2}
}
\providecommand{\href}[2]{#2}

\end{document}